\documentclass[12pt]{amsart}
\usepackage{xcolor}
\usepackage[utf8]{inputenc}
\newcommand{\lr}[1]{\left(#1\right)}
\newcommand{\abs}[1]{\left\vert #1 \right\vert}
\newcommand{\norm}[2]{\left\Vert #1 \right\Vert _{#2}}
\newcommand{\set}[1]{\left\{#1\right\}}
\newcommand{\scalar}[2]{\langle #1,#2\rangle}
\newcommand{\real}{\mathbb R}
\newcommand{\nat}{{\mathbb N}}
\newcommand{\gd}{g^\delta}
\newcommand{\fd}{f^\delta}
\newcommand{\fad}{f^\delta_\alpha}
\newcommand{\bsubf}{b_f}
\newcommand{\yd}{y^\delta}

\newcommand{\Lmu}{L^2(S,\Sigma,\mu)}
\newcommand{\E}{\mathbb E}
\newcommand{\Pa}{\Phi_\alpha}
\newcommand{\tPa}{\tilde\Phi_\alpha}
\newcommand{\Ra}{R_\alpha}
\newcommand{\tRa}{\tilde R_\alpha}
\newcommand{\Aop}{\operatorname{A}}
\newcommand{\essinf}{\operatorname{essinf}}
\newcommand{\bast}{b_{\ast}}
\newcommand{\boast}{b^{\ast}}
\newcommand{\phast}{\varphi_{\ast}}

\newtheorem*{fact}{Fact}
\newtheorem{prop}{Proposition}
\newtheorem{cor}{Corollary}
\newtheorem{lem}{Lemma}
\newtheorem{de}{Definition}
\newtheorem{ass}{Assumption}
\theoremstyle{definition}
\newtheorem{rem}{Remark}
\newtheorem{xmpl}{Example}

\def\a{\alpha}
\def\b{\beta}
\def\d{\delta}

\def\l{\lambda}

\def\O{\Omega}

\def\m{\mu}

\def\t{\tau}

\def\R{{\mathbb{R}}}
\def\N{{\mathbb{N}}}

\title[Regularization of multiplication operators]{Regularization of
  linear ill-posed problems involving  multiplication operators}
\author{Peter Math\'e}
\address{Weierstra{\ss} Institute for Applied Analysis and
Stochastics, Mohren\-straße 39, 10117 Berlin, Germany}
\email{peter.mathe@wias-berlin.de}
\author{M.~Thamban Nair}
\address{Department of Mathematics, IIT Madras, Chennai 600036, India}
\email{mtnair@iitm.ac.in}
\author{Bernd Hofmann}
\address{Department of Mathematics, Chemnitz University of Technology,
  09107 Chemnitz,  Germany}
\email{hofmannb@mathematik.tu-chemnitz.de}
\thanks{BH was supported by German Research Foundation (DFG-grant HO 1454/12-1).}

\date{Version: \today}
\begin{document}
\begin{abstract} 
  We study regularization of ill-posed equations involving
  multiplication operators when the multiplier function is
  positive almost everywhere and zero is an accumulation point of the
  range of this function.
Such equations naturally arise from equations based on non-compact self-adjoint operators in Hilbert space,
after applying unitary transformations arising out of the spectral
theorem.
For classical regularization theory, when noisy
  observations are given and the noise is deterministic and bounded,
  then non-compactness of the ill-posed equations is a minor
  issue. However, for statistical ill-posed equations with non-compact
  operators less is known if the data are blurred by white noise.
  We develop a regularization theory with
  emphasis on this case. In this context, we highlight several aspects,
  in particular we discuss the intrinsic degree of ill-posedness in terms of
  rearrangements of the multiplier function. Moreover, we address the required
  modifications of classical regularization schemes in order to be
  used for non-compact statistical problems, and we also introduce the
  concept of the effective ill-posedness of the operator equation
  under white noise. This study is concluded with prototypical
  examples for such equations, as these are deconvolution equations
  and certain final value problems in evolution equations.
\end{abstract}
\maketitle

\section{Introduction, background}
\label{sec:intro}
This study is devoted to multiplication operators in the context of ill-posed
linear operator equations
\begin{equation}
  \label{eq:opeq}
   \Aop\, x \,=\, y
 \end{equation}
with bounded self-adjoint positive operators~$\Aop\colon H \to H$ mapping in the  (separable) Hilbert
space~$H$ and possessing a non-closed range $\mathcal{R}(A)$.
We fix a measure space~$\lr{S,\Sigma,\mu}$ on the set $S$ with a
\emph{$\sigma$-finite} measure~$\mu$ defined on the $\sigma$-algebra $\Sigma$, and consider instead of \eqref{eq:opeq} the equation
\begin{equation}
\label{eq:opeqmult}
 b(s) f(s) = g(s),\quad s\in S
 \end{equation}
in the setting of the Hilbert space~$L^{2}(S,\Sigma,\mu)$. In
particular, the equation above is to hold $\mu$-almost everywhere ($\mu$-a.e.). This multiplication setup is prototypical
based on the following version of the {\it Spectral Theorem}.
\begin{fact}[see e.g.~~\cite{MR0150600} for a detailed discussion]
For every bounded self-adjoint
operator~$\Aop\colon H\to H$ there is a $\sigma$-finite measure
space~$\lr{S,\Sigma,\mu}$, a real-valued essentially bounded function~$b\in
L^{\infty}(S,\Sigma,\mu)$ and an isometry  $U\colon H \to \Lmu$
such that~$U\Aop U^{-1} = M_{b}$, where $M_b$  is
the  multiplication operator, assigning $f\in\Lmu \mapsto b\cdot
f\in\Lmu$.
\end{fact}

Since within this study the
linear operator~$\Aop$ is assumed to be bounded  self-adjoint and
positive, we have
a constant $\overline b>0$ such that the multiplier function
$b$ in \eqref{eq:opeqmult} obeys the inequalities
$0<b(s) \le \overline b<\infty$ for almost all $s \in S$. However,
as a consequence of the ill-posedness of equation \eqref{eq:opeq}, which
implies that zero is an accumulation point of the spectrum of $\Aop$, the
function $b$ must have essential zeros, which means that $\essinf_{s \in S} b(s)=0$.

So, the recovery of the element $x \in H$ in \eqref{eq:opeq} from noisy
data
\begin{equation}
\label{eq:ill-posed}
   \yd := \Aop x + \delta\eta
 \end{equation}
of the right-hand side $y$ carries over to the reconstruction of the
solution $f(s),\, s\in S,$ of equation \eqref{eq:opeqmult} from noisy data
\begin{equation}
   \label{eq:ill-posed-multiplier}
\gd:= U \yd = b \cdot (U x) + \delta (U \eta)
\end{equation}
of the right-hand side $g= U(A x)$. The variable $\eta$ turns to the
noise~$\xi := U\eta$, and further properties will be given in
Definitions~\ref{de:det-noise} and~\ref{de:white-noise} below. The
analysis will be different for bounded deterministic noise and for
statistical white noise.

Thus, we consider the reconstruction of the function $f$
in the Hilbert space~$L^{2}(S,\Sigma,\mu)$, from the knowledge of the noisy data
\begin{equation}
  \label{eq:base-multiplication}
  \gd(s): = b(s) f(s) + \delta \xi(s),\quad s\in S,
\end{equation}
where $b\in L^\infty(S,\Sigma,\mu)$ is given, and we assume that $\delta>0$ denotes the noise level.

Statistical inverse problems under {\it  white noise} and the reduction to
multiplication problems as~(\ref{eq:base-multiplication}) were
discussed in~\cite{MR2211346}. Multiplier equations as
in~(\ref{eq:opeqmult}) were studied
in \cite{MR1701353} from the regularization
point of view for  $S=(0,1)$ and~$\mu$ being the Lebesgue
measure, which we throughout will denote by $\lambda$.

The outline of the remainder of this paper is as follows: In Section~\ref{sec:assumptions} we shall
describe the framework for the analysis, and we will also provide
several auxiliary results that might be of interest.
Section~\ref{sec:error} is devoted to the error analysis. The main
results, presented in Propositions~\ref{est-derministic} \&~\ref{est-white}, yield estimations from above of the regularization error under
bounded deterministic and white noise, respectively. Finally,
Section~\ref{sec:op-eqns} exhibits that the well-known (and typically
non-compact)
deconvolution and final value problems fit the considered framework
after turning from operator equations to the current setup by means of
the Fourier transform.

\section{Notation and auxiliary results}
\label{sec:assumptions}

In this section we shall first discuss the impact of properties of the
multiplier function~$b$ on the intrinsic difficulty of the inverse
problem. Then we turn to introducing the concept of solution smoothness. Finally, we
introduce the concept of regularization schemes.

\subsection{Degree of ill-posedness incurred by the multiplier
  function~$b$}
\label{sec:degree}

As was mentioned in Section~\ref{sec:intro} ill-posedness of the
multiplication problem (\ref{eq:opeqmult}) is a consequence of having zero as accumulation
point of the essential range of $b$. For the character of ill-posedness, however, the
location of the essential zeros of the function~$b$ should not be relevant. Therefore, a normalization of the function~$b$
in~(\ref{eq:opeqmult}) is desirable. In this context,
the increasing rearrangement of~$b$ was
considered in the study~\cite{MR1701353}.  For such setting we let~$\boast$ denote the
increasing rearrangement
$$
\boast (t) := \sup\set{\tau: \; \mu(\set{s: \, b(s) \leq \tau}) \leq
  t},\quad t>0.
$$
However, this approach is limited to underlying
$S$ and $\mu$ with finite measure values $\mu(S)$.

Another normalization is the decreasing
rearrangement~$\bast$ of the multiplier function~$b$, which is based
on the distribution function~$d_b$,  defined by  $d_b(t):= \mu\{s\in
S: b(s)>t\}$ for $t>0$. Then we let the  decreasing
rearrangement of $b$  be given as
$$
 \bast(t):=  \inf\set{\t>0: d_b(\t)\leq  t},\quad  0<t< \m(S).
$$
Note that ~$\bast$ is defined on~$[0,\mu(S))$, equipped with the
  Lebesgue measure $\lambda$.
  In the context of ill-posed equations this normalization was first
  used in~\cite{MR1264058}.

We also notice that for infinite measures, that is, for~$\mu(S)=\infty$, the function~$\bast$
may have infinite value.  Therefore we confine to the case of $b$ satisfying the following assumption.
\begin{ass}\label{ass:distfunction}
  If~$\mu(S)=\infty$ then the function~$b$ is assumed to \emph{vanish
    at infinity},
in the sense that
$\mu\{s\in S: b(s)>t\}$ is finite for every $t>0$.
\end{ass}
\begin{rem}
  Assumption~\ref{ass:distfunction} is important to guarantee the
  existence of a decreasing rearrangement~$\bast$ which is
  equimeasurable with~$b$, meaning that it has the same distribution
function as~$b$, i.e.,
$$
\lambda(\set{\t:\, \bast(\t) >t}) =
\mu(\set{s\in S:\, b(s) >t}),\quad  t>0.
$$
The characterization of cases when the function~$\bast$ is
equimeasurable with~$b$ was first given by
Day
in~\cite{MR2619530}, and for
infinite measures~$\mu(S)=\infty$ the
Assumption~\ref{ass:distfunction} is known to be sufficient to
guarantee this, see~\cite[Chapt.~VII]{MR0372140} for details.
For calculus with decreasing rearrangements we also refer
to~\cite[Chapt.~2]{MR928802}.
Moreover, functions vanishing at infinity are important in Analysis,
see~\cite[Chapt.~3]{MR1817225}.
\end{rem}

For the subsequent analysis we shall first assume
that our focus is on the Lebesgue measure $\mu=\lambda$,
either on~$[0,\infty)$ for the decreasing
rearrangement, or on some bounded interval~$[0,a]$ for the increasing
rearrangement. In such case, $\Sigma$ denotes the corresponding Borel
$\sigma$-algebra. In fact, one may take $a:=\|b\|_{\infty}$.
    The corresponding analysis extends to measures~$\mu\ll \lambda$
    with density~$\frac{d\mu}{d\lambda}$ which obey~$0 < \underline c
\leq \frac{d\mu}{d\lambda} \leq \bar C < \infty$. If this is the case
then it is easily seen that the increasing rearrangements~$b_\mu^\ast$
and~$b_\lambda^\ast$ of~$b$ corresponding to~$\mu$ and~$\lambda$,
respectively, satisfy
$$
b_\mu^\ast(\underline c t) \leq b_\lambda^\ast(t) \leq b_\mu^\ast(\bar
C t),\quad t>0,
$$
Similar argumants apply to the decreasing rearrangement. Thus, the
asymptotic results as these will be established for the Lebesgue
measure~$\lambda$ find their counterparts for other measures~$\mu$.

The first observation concerns the decreasing
rearrangement. 
For $M>0$ we assign the truncated
function~$\tilde b_{M}(t) := b(t) \chi_{(M,\infty)}(t)$, and the
shifted (to zero) version~$b_{M}(t) := \tilde b_{M}(t+M)$.

\begin{prop}
  We have that
$$ \lr{b_{M}}_{\ast}(t) = \lr{\tilde b_{M}}_{\ast}(t)
\leq \bast(t),\quad t>0.
  $$
\end{prop}
\begin{proof}
The equality   is a result of the translation invariance of the
  Lebesgue measure, and the inequality
  follows from the fact
  that the decreasing rearrangement is order-preserving.
\end{proof}
Thus, the decreasing rearrangement does not take into account any
zeros which are present on bounded domains. Only the behavior at
infinity is reflected.

For the increasing rearrangement we shall follow a constructive
approach. Here we shall assume that the function~$b$ is piece-wise
continuous and has finitely many zeros. Specifically, we assume a
representation
\begin{equation}
  \label{eq:b-representation}
  b(s) = 
  \sum_{j\in A^{+}} b_{j}(s- s_{j}) + \sum_{j\in A^{-}}
  b_{j}(s_{j} -s),\quad 0 \leq s \leq a,
\end{equation}
where
\begin{enumerate}
\item[(i)]\label{it:zeros} the reals~$s_{j}\,(j=1,\dots, m)$ are the
  (distinct) locations of the zeros,
\item[(ii)]\label{it:bj} for each~$j=1,\dots,m$ we have that~$b_{j}(0)=0$,
  and there is a neighborhood of
  zero~$[0,a_{j})$ such that
  \begin{itemize}
  \item $b_{j}\colon [0,a_{j}]\to \real^{+},\
    j=1,\dots,m$ is \emph{continuous and strictly increasing}. 
  \end{itemize}
 The function~$\bar b$ satisfies~$\essinf \bar b>0$.
\item[(iii)]\label{it:disjoint} The set~$A= \set{1,\dots,m} = A^{+}\sqcup
  A^{-}$ is decomposed into two disjoint
  subsets~$A^{+}$ and~$A^{-}$, possible empty, and
\item[(iv)]\label{it:domination} there is one function, say~$b_{k}$ such that its
  inverse~$b_{k}^{-1}$ dominates\footnote{We say that a (non-negative) function~$g$
    dominates~$f$, and write $f \preceq g$, if there are a neighborhood~$[0,\varepsilon)$ and a
    constant~$k>0$ such that~$f(t) \leq k g(t),\ 0 \leq t \leq
    \varepsilon$.}  all other functions~$b_{j}^{-1}$, i.e.,\
  $b_{j}^{-1} \preceq b_{k}^{-1}$.
\end{enumerate}
Thus, the function~$b$ is a superposition of a function bounded away
from zero, and of increasing and
decreasing parts. We also stress that domination as in
item~(iv) does not extend from functions~$f^{-1},g^{-1}$ to the
inverse functions~$f,g$ unless additional assumptions are made, we
refer to~\cite{MR2442065} for a discussion.

Under the above assumptions we state the following result.
\begin{prop}
  Let the function~$b$ be as in the equation~(\ref{eq:b-representation}),
    and let~$b_{k}$ be the function from item~(iv) above.
  Then there is a
  constant~$C\geq 1$ such that
  $$
  b_{k}(s) = \lr{b_{k}}^{\ast}(s) \geq \boast(s) \geq
  \lr{b_{k}}^{\ast}\lr{\frac{s}{C m}} = b_{k}\lr{\frac{s}{C
      m}},
  $$
  for sufficiently small~$s> 0$.
\end{prop} 
\begin{proof}
  Clearly, the function~$b_{k}$ is increasing near zero, such
  that it coincides with its increasing rearrangement, which explains
  the outer equalities.

  To establish the inner inequalities we argue as follows.
  Recall that we need to control
  \begin{align*}
    \lambda(b\leq \tau) = \lambda\lr{ \bar b(s) + \sum_{j\in A^{+}} b_{j}(s- s_{j}) + \sum_{j\in A^{-}}
  b_{j}(s_{j} -s) \leq \tau}.
  \end{align*}
  If~$\tau>0$ is small enough then, by item~(ii) the
  contribution of~$\bar b$ is neglected, and the sub-level sets $\{b_{j}\leq
  \tau\}\,(j=1,\dots,m)$ are disjoint intervals, which in turn yields
  \begin{align*}
   \lambda(b\leq \tau) & = \sum_{j\in A^{+}} \lambda(b_{j}(s- s_{j})
    \leq \tau) +  \sum_{j\in A^{-}} \lambda(b_{j}(s_{j} - s)    \leq
                         \tau) \\
    & = \sum_{j\in A^{+}} b_{j}^{-1}(\tau) + \sum_{j\in A^{-}}
      b_{j}^{-1}(\tau)\\
    &= \sum_{j\in A}b_{j}^{-1}(\tau).
  \end{align*}
  By the domination assumption from item~(iv) we find
  a constant~$C\geq 1$ such that
  \begin{equation}
    \label{eq:sum-bk}
    b_{k}^{-1}(\tau) \leq \sum_{j\in A}b_{j}^{-1}(\tau) \leq C m b_{k}^{-1}(\tau).
  \end{equation}
  Now, asking for the sup over all~$\tau>0$ such that~$\lambda(b\leq
  \tau)\leq s$ we find that
$$
b_{k}(s) \geq \boast(s) \geq
b_{k}\lr{\frac{s}{mC}},
$$
  for sufficiently small~$s> 0$.
This completes the proof.
\end{proof}
The above proposition asserts (heuristically)  that the part in the
decomposition of~$b$ in~(\ref{eq:b-representation}) which has the
highest order zero determines the asymptotics of the increasing rearrangement.

\subsection{Solution smoothness}
\label{sec:smoothness}

In order to quantify the error bounds, we need to specify the way in
which the solution smoothness will be expressed. This is given in
terms of source conditions based on index functions.
Here, and throughout, by an \emph{index function}, we mean a
strictly increasing continuous function $\varphi: (0, \infty)\to [0,
\infty)$ such that $\lim_{t \to +0}\varphi(t)=0$.
  \begin{de}[source condition]\label{de:source-condition}
    A function~$f\in\Lmu$ obeys a source condition with respect to the
    index function~$\varphi$ and the multiplier function~$b$ if
    $$
    f(s) = [\varphi(b)](s) v(s) := \varphi(b(s)) v(s),\quad s\in
    S,\quad \mu-a.e.
    $$
    with~$\norm{v}{\Lmu}\leq 1$.
  \end{de}
  \begin{rem}
Here, we briefly discuss the meaning of source conditions as
given in Definition~\ref{de:source-condition}. Let the operator~$A$
and the multiplication operator $M_{b}$ be related as described in the introduction. Then it is clear from the
relation between the noisy data representations~(\ref{eq:ill-posed})
and~(\ref{eq:ill-posed-multiplier}), that a source condition
$f= Ux = \varphi(b) v,\ \norm{v}{\Lmu}\leq 1$ yields the
representation~$x = \varphi(A) w$ with~$w:= U^{-1}v\in H$
and~$\norm{w}{H} = \norm{v}{\Lmu}\leq 1$. This is the standard form of
general smoothness in terms of source conditions, given with respect to
the forward operator~$A$ from~(\ref{eq:ill-posed}),
see~\cite{MR1984890,MR2521497}.
\end{rem}

  It is interesting to relate this to `classical' smoothness in the
  sense of Hilbertian Sobolev spaces $H^p(S,\mu)$. Precisely,
for smoothness parameter~$p>0$ we let~$H^{p}(S,\mu)$ be the Hilbert space of all functions~$f\colon
S \to \real$ such that
$$
\norm{f}{p} := \left(\int_{S} \abs{f(s)}^{2} \lr{1 + \abs{s}^{2}}^{p} \;
d\mu(s)\right)^{1/2} < \infty.
$$
The question is, under which conditions this type of smoothness can be expressed in
terms of source conditions as in Definition~\ref{de:source-condition}.
We start with the following technical result.
\begin{lem}\label{lem:phast}
  Suppose that~$\mu(S)=\infty$ and that the function~$b$ vanishes
  at infinity (cf.~Assumption~\ref{ass:distfunction}). Moreover, let
  there exist positive constants $M<\infty$ and $c>0$ such that
  $$b(s)\geq c, \quad \mbox{for} \quad \abs{s}\leq M$$
  and
    \begin{equation}
    \label{eq:mu-condition}
    \mu\lr{\set{x,\ b(x) >b(s)}}\asymp \abs{s},\quad \text{for}\
    \abs{s} >M.
  \end{equation}
  Then the function~$\phast$, given for sufficiently small $t>0$ as
\begin{equation}
  \label{eq:phast}
  \phast(t):= \frac{1}{\mu\lr{\set{x,\ b(x) > t}}},
\end{equation}
constitutes an index function. Moreover, we have
the asymptotics
\begin{equation}
  \label{eq:phast-asymptotics}
  \phast(b(s)) \asymp \frac 1 {\abs{s}}\quad \mbox{as} \quad  \abs{s}\to \infty.
\end{equation}
\end{lem}

\begin{xmpl}[power-type decay on $[0,\infty)$]
    For $\kappa>0$ we consider the
functions $f \in L^2([0,\infty),\mu)$ with Lebesgue measure $\mu$ defined as
  $$
  b(s):=\frac{1}{1+s^{1/\kappa}}, \qquad 0 \le s < \infty.
  $$
  Then the assumptions of Lemma~\ref{lem:phast} are fulfilled, and for
  sufficiently small $t>0$, we have in this case 
$$\mu\lr{\set{x,\ b(x) > t}}=\left(\frac{1-t}{t}\right)^{\kappa}$$
and hence $\phast(t) \asymp t^{\kappa}$ as $t \to +0$.
\end{xmpl}

\begin{cor}
  Under the assumptions of Lemma~\ref{lem:phast} consider the
  function~$\phast$ as in~(\ref{eq:phast}).
  The function~$f$ belongs to~$H^{p}(S,\mu)$ if and only if it obeys a
  source condition with respect to (a multiple of) the
  function~$\phast^{p}(t),\ t>0$.
\end{cor}

\begin{proof}
  Under the assumptions of Lemma~\ref{lem:phast} for the
  function~$b$  there are constants~$0 < c < C <
  \infty$ such that
  \begin{equation}
    \label{eq:lower-upper-bound}
    c \leq \inf_{s\in S} \lr{1 + \abs{s}^{2}} \phast^{2}\lr{b(s)}
    \leq \sup_{s\in S} \lr{1 + \abs{s}^{2}} \phast^{2}\lr{b(s)} \leq C.
  \end{equation}
  This is easily seen for~$\abs{s} \leq M$, as given in
  Lemma~\ref{lem:phast}. For~$\abs{s} > M$ we see that
  $$
  \lr{1 + \abs{s}^{2}} \phast^{2}\lr{b(s)} \asymp \lr{1 +
    \abs{s}^{2}} \abs{s}^{-2}.
  $$
  But for~$\abs{s}\geq M$ we have that~$1 \leq \lr{1 +
    \abs{s}^{2}} \abs{s}^{-2} \leq \lr{1 +
    M^{2}} M^{-2}$, where the right hand side bound follows from the
  monotonicity of ~$x\mapsto (1+x)/x,\ x>0$, and this
  proves~(\ref{eq:lower-upper-bound}).
 Now, suppose that~$f\in  H^{p}(S,\mu)$. Consider the element~$w(s):=
 \frac{f(s)}{\phast^{p}(b(s))}$, where~$\phast$ is as above. It is
 enough to show that~$w\in L^{2}(S,\mu)$, i.e.,\ that it serves as a
 source element. We have that
 \begin{align*}
\int_{S}\abs{w(s)}^{2}\, \mu(s) & \leq  \int \abs{f(s)}^{2} \lr{1 +
                                  \abs{s}^{2}}^{p}\, d\mu(s) \sup_{s\in S} \frac
                                  1 {\lr{1 +
                                  \abs{s}^{2}}^{p}\abs{\phast^{p}(b(s))}^{2}}
   \\
   &= \norm{f}{p}^{2} \sup_{s\in S} \frac
                                  1 {\left[\lr{1 +
                                  \abs{s}^{2}}\abs{\phast^{2}(b(s))}\right]^{p}},
 \end{align*}
 and the latter is finite by~(\ref{eq:lower-upper-bound}). On the
 other hand, under a source condition for~$f$ we can bound
 \begin{align*}
   \int_{S}\abs{f(s)}^{2}\lr{1 + \abs{s}^{2}}^{p}\,d\mu(s) &\leq
\norm{w}{\Lmu}\sup_{s\in S} \left[\lr{1 + \abs{s}^{2}} \abs{\phast^{2}(b(s))}\right]^{p},
 \end{align*}
 where the  supremum is again bounded
 by~(\ref{eq:lower-upper-bound}). The proof is complete.
\end{proof}

\subsection{Regularization}
\label{sec:regularization}

For reconstruction of $f(s),\, s\in S$, we shall use
\emph{regularization schemes} $\Pa \colon [0,\infty) \to \real^{+}$,
parametrized by $\alpha >0$, see e.g.~\cite{MR1984890}.

\begin{de}[regularization scheme]\label{de:regularization}
  A family~$(\Pa)$ of real valued Borel-measurable functions $\Pa(t),\ t\geq 0,
  \alpha>0$, is called a regularization if
 there are constants $C_{-1}>0$ and $C_0 \ge 1$ such that 
  \begin{enumerate}
  \item[(I)]\label{it:to0} for each~$t>0$ we have ~$ t \Pa(t) \to 1 $ as~$\alpha \to 0$,
  \item[(II)]\label{it:1alpha-estimate} $\abs{\Pa(t)}\leq
    \frac{C_{-1}}{\alpha}$ for $\alpha >0$, and
  \item[(III)]
the function $R_\alpha(t):= 1-t\Phi_\alpha(t)$, which is called a \emph{residual function}, satisfies ~$\abs{R_{\alpha}(t)}\leq C_{0}$ for all~$t\ge 0$ and $\alpha>0$.
\end{enumerate}
\end{de}

For the case of statistical noise, additional assumptions have to be
made. These will be introduced and discussed later.

We apply a regularization~$(\Pa)$ to a function~$b$ in the way
$$
[\Pa(b)](s) := [\Pa \circ b](s) = \Pa(b(s)),\quad s\in S.
$$

Having chosen a regularization~$\Pa$, and given data~$\gd$ we consider
the function
\begin{equation}
  \label{eq:appr-solution}
  \fad(s) := [\Pa(b)](s) \gd(s),\quad s\in S,
\end{equation}
or in short~$\fad := \Pa(b) \gd$, as a candidate for the approximate solution.

For the subsequent error analysis the following property of a
regularization proves important, again we refer to~\cite{MR1984890, NPT}.

\begin{de}
  [qualification]\label{de:qualification}
Let~$\varphi$ be any index function.
 A regularization~$(\Pa)$ is said to have qualification~$\varphi$ if
  there is a constant~$C_{\varphi}>0$ such that
  $$
  \sup_{t\geq 0}\abs{R_\a(t)} \varphi(t) \leq C_{\varphi} \varphi(\alpha),\quad \alpha>0.
  $$
\end{de}

\begin{xmpl}[spectral cut-off]\label{xmpl:cut-off-def}
  Let the regularization be given as
  $$
  \Pa^{\text{c-o}}(t) :=
  \begin{cases}
    \frac 1 t, & t > \alpha\\
    0, & else.
  \end{cases}
  $$
  This obeys the requirements of regularization with~$C_{-1}=1$. It
  has arbitrary qualification. That is, for any index function, the
  requirement in Definition \ref{de:qualification} will be satisfied. The
  corresponding residual function is~$\Ra = \chi_{\set{t: \ t \leq
      \alpha}}$, so that for a function $b$ on $S$,
   ~$\Ra(b)= \chi_{\set{s:\ b(s) \leq \alpha}}$. 
\end{xmpl}
\begin{xmpl}  [Lavrent'ev regularization]
  This method corresponds to the function
  $$
  \Pa(t) := \frac 1 {t + \alpha},\quad t \ge 0,\,\alpha>0.
  $$
  Lavrent'ev regularization is known to have at most `linear' qualification.
  More generally,
  index functions~$\varphi(t):= t^{\nu},\ t>0,$ are qualifications
  whenever the exponent~$\nu$ satisfies~$0< \nu \leq 1$.
\end{xmpl}

For infinite measures~$\mu$, and under white noise, it will be seen
that it is important that the regularization~$\lr{\Pa}$ will vanish
for small~$0 \le t\leq \alpha$ This is formalized in
\begin{ass}\label{ass:finite-condition}
 For each~$\alpha>0$ the function~$\Pa$ vanishes on the set~$\set{t \ge 0: \
   t\leq \alpha}$.
\end{ass}
This assumption holds true for spectral cut-off, but
it is not fulfilled for most other regularizations. However, we can
modify any regularization to obey Assumption~\ref{ass:finite-condition}.

\begin{lem}
  Let~$\lr{\Pa}$ be any regularization with constants~$C_{-1}$ and~$C_{0}$. Assign
  $$
  \tPa(t) :=   \chi_{(\alpha,\infty)}(t) \Pa(t),\quad t>0.
  $$
  Then~$\lr{\tPa}$ is a regularization scheme with same
  constants~$C_{-1}$ and~$C_{0}$.
  Moreover, an index function~$\varphi$ is a qualification
  of~$\lr{\Pa}$ if and only if it is a qualification of~$\lr{\tPa}$ with
constant~$\tilde C_{\varphi} =
  \max\set{C_{\varphi},C_{0}}$.
\end{lem}
\begin{proof}
  We verify the properties. For~$t>\alpha$ the regularizations~$\tPa$
  and~$\Pa$ coincide, thus item~(I) holds true. Also,
  $\abs{\tPa(t)} \leq \abs{\Pa(t)}$, such that we can let~$\tilde
  C_{0}:= C_{0}$. Next, it is easy to check that~$\tRa(t) =
  \chi_{(a,\infty)}(t) \Ra(t) + \chi_{(0,a]}(t)$, which allows us to
  prove the second assertion, after recalling that~$C_{0}\geq 1$.

  Finally,  we bound~$\abs{\tRa(t)}\varphi(t)$. Plainly, if~$t \leq
  \alpha$ then
  $$
  \abs{\tRa(t)}\varphi(t)\leq C_{0} \varphi(\alpha).
  $$
  Otherwise, for~$t>\alpha$ both functions~$\tRa$ and~$\Ra$ coincide,
  This completes the proof.
\end{proof}
Therefore, we may tacitly assume that the regularization of choice is
accordingly modified to meet Assumption~\ref{ass:finite-condition}.

\section{Error analysis}
\label{sec:error}

As stated in the introduction, we shall discuss error bounds, both for
the classical setup of bounded deterministic noise, as well as for
statistical white noise, to be defined now.
\begin{de}
  [deterministic noise]\label{de:det-noise}
  The noise term~$\xi = \xi(s)$ is norm-bounded by one, i.e., \
  $\norm{\xi}{\Lmu}\leq 1$.
\end{de}
We shall occasionally adopt the notation~$\xi_s:=\xi(s)$.
\begin{de}
  [white noise]\label{de:white-noise}
  There is some probability space~$(\Omega,\mathcal F,P)$ such that the
  family~$\set{\xi_{s}}_{s\in S}$ constitutes a centered
  stochastic process\footnote{For each~$s\in S$ we have a random
    variable~$\xi_{s}\colon \Omega \to \real$.} with~$\E\xi_{s}=0$ for
  all~$s\in S$, and $\E\abs{\xi_{s}}^{2} =1,\ s\in S$.
\end{de}

We return to the noisy equation~(\ref{eq:base-multiplication}).
Writing (the unknown)~$  f_\a(s) := [\Pa(b)](s) g(s),\quad s\in S,$
by (\ref{eq:base-multiplication}) and (\ref{eq:appr-solution}), we obtain
\begin{eqnarray*}
f - \fad &=& [f - \Pa(b)(g)] - [\Pa(b)(g^\d) - \Pa(b)(g)]\\
&=& [f - \Pa(b)(bf)] - [\Pa(b)(g^\d) - \Pa(b)(g)]\\
& = &     [I - \Pa(b)b]f  -\delta \Pa(b) \xi.
\end{eqnarray*}
Thus, we have the decomposition of the  error of the reconstruction~$\fad$ in a natural
way, by using the residual function~$\Ra$, as
\begin{equation}
  \label{eq:error-decomposition}
  f - \fad = 
  \Ra(b) f -  \delta \Pa(b) \xi,
\end{equation}
The term~$\Ra(b) f $ is completely deterministic, the noise properties
are inherent in~$\Pa(b) \xi$, only.

\subsection{Bounding the bias}
\label{sec:bias-bound}
The  (noise-free) term~$\Ra(b)f$ in the
decomposition~(\ref{eq:error-decomposition}) gives rise to the bias,
defined as
$$
\bsubf(\alpha) := \norm{\Ra(b)f}{\Lmu},
$$
and it is called the \emph{profile
  function} in~\cite{MR2318806}.
We shall assume that the solution admits a
source condition as in Definition~\ref{de:source-condition}, and that
the chosen regularization has this as a qualification
so that
\begin{multline}
  \label{eq:triangle-bound-1}
  \norm{\Ra(b) f}{\Lmu} \leq \norm{\Ra(b)\varphi(b(s))v(s)}{\Lmu}\\
  \leq \norm{\Ra(b) \varphi(b(s))}{\infty}\norm{v}{\Lmu}\leq C_{\varphi}\varphi(\alpha).
\end{multline}
\medskip

We briefly highlight the case when~$\mu$ is a finite measure. It is to
be observed that
if ~$f\in L^{\infty}(S,\mu)$, then
\begin{equation}\label{eq:2infty}
\norm{\Ra(b) f}{\Lmu} \leq \norm{\Ra(b)}{\Lmu} \norm{f}{L^{\infty}(S,\mu)}.
\end{equation}
\begin{xmpl}
For Lavrent'ev regularization,
spectral cut-off, and with function~$b(s) := s^{\kappa},\ s>0$,
with~$\kappa>0$, we see that
\begin{align*}
  \norm{\Ra(b)}{\Lmu}^{2} =
  \begin{cases}
    \alpha^{2}\int (\alpha + s^{\kappa})^{-2}d\mu(s), & \text{for
      Lavrent'ev regularization}\\
    \mu\lr{\set{s:\ s^\kappa \leq \alpha}}, & \text{for spectral cut-off}.
  \end{cases}
\end{align*}
From this, we conclude that for Lavrent'ev regularization the bound
in~(\ref{eq:2infty}) is finite only
if~$\kappa >1/2$, whereas for spectral cut-off this holds for
all~$\kappa>0$. If $\mu$ is the Lebesgue measure $\lambda$ on $(0,1)$, then we
find that
$$\|\Ra(b)\|_{L^2((0,1),\Sigma,\lambda)} \leq C \alpha^{1/(2\kappa)}$$
in either case.

A similar
bound, relying on Tikhonov regularization was first given in~\cite[ Theorem~4.5]{MR1701353}.
\end{xmpl}

\subsection{Error under deterministic noise}
\label{sec:deterministic}
Although the focus of this study is on statistical ill-posed problems,
we briefly sketch the corresponding result for bounded deterministic
noise as introduced in Definition~\ref{de:det-noise}.
In this case we bound the error, starting from the
decomposition~(\ref{eq:error-decomposition}),  by using the triangle
inequality, for~$\alpha>0$ as
\begin{equation}
  \label{eq:triangle-bound}
  \norm{f - \fad}{\Lmu} \leq \norm{\Ra(b) f}{\Lmu} + \delta
  \norm{\Pa(b) \xi}{\Lmu}.
\end{equation}
Now, using the item~(2) in Definition \ref{de:regularization}, we obtain a bound for the noise term as
$$
\delta   \norm{\Pa(b) \xi}{\Lmu} \leq \delta \sup_{s\geq 0}
\abs{\Pa(b(s))} \leq  C_{-1}\frac\delta {\alpha},\quad \alpha>0.
$$
This together with the estimate (\ref{eq:triangle-bound-1}) for the noise free term gives the following
\begin{prop}\label{est-derministic}
  Suppose that the solution $f$ satisfies the source condition as in   Definition~\ref{de:source-condition}, and that  a
  regularization~$(\Pa)$ is chosen with  qualification~$\varphi$. Then
  $$
  \norm{f - \fad}{\Lmu} \leq C_{\varphi}\varphi(\alpha) +
  C_{-1}\frac\delta {\alpha},\quad \alpha>0.
  $$
  The \emph{a priori parameter choice}~$\alpha_{\ast}=
  \alpha_{\ast}(\varphi,\delta)$ from solving the equation
  \begin{equation}
    \label{eq:a-priori-choice}
    \alpha \varphi(\alpha) = \delta
  \end{equation}
  yields the error bound
  \begin{equation}
    \label{eq:a-priori-error}
   \norm{f - \fad}{\Lmu}  \leq 2\max\set{C_{\varphi},C_{-1}}\varphi(\alpha_{\ast}),
 \end{equation}
 uniformly for functions~$f$ which obey a source condition with
 respect to the index function~$\varphi$.
\end{prop}

\subsection{Error under white noise}
\label{sec:white-noise}

Here we assume that the underlying noise is as in
 Definition~\ref{de:white-noise}. Thus, since $\xi$ is a random variable, it is a function of $\omega\in \Omega$ so that $f_\a^\d$ also a function of $\omega\in \Omega$. Hence, for each fixed $\omega\in \O$, from~(\ref{eq:error-decomposition}) we obtain
\begin{eqnarray}
\norm{f - \fad(\omega)}{\Lmu}^{2}  &=&  \norm{\Ra(b) f}{\Lmu}^{2} +
2 \delta\scalar{\Ra(b) f}{\Pa(b) \xi}\label{eq:err-deco-white-noise}\\
&& + \delta^{2}\int_{S}\abs{\Pa(b(s))}^2\abs{\xi_{s}(\omega)}^{2}\;d\mu(s).\nonumber
\end{eqnarray}
The error of the regularization~$\Pa$ under white noise is
measured in RMS sense, that is,\ it is defined as
\begin{equation}
  \label{eq:rms-error}
  e(f,\Pa,\delta)^{2} := \E\norm{f - \fad}{\Lmu}^{2},
\end{equation}
where the expectation is with respect to the probability~$P$ governing
the noise process.
From the properties of the noise we deduce
from~(\ref{eq:err-deco-white-noise}) the \emph{bias--variance decomposition}
\begin{equation}
  \label{eq:bias-variance-deco}
  \E\norm{f - \fad}{\Lmu}^{2} = \norm{\Ra(b) f}{\Lmu}^{2} + \delta^{2}
  \E \int_{S}\abs{\Pa(b(s))}\abs{\xi_{s}(\omega)}^{2}\;d\mu(s).
\end{equation}
The first summand above, the squared bias, is treated as
in~\S~\ref{sec:bias-bound}. It remains to bound the variance, that is,
the second summand in~(\ref{eq:bias-variance-deco}). By interchanging
expectation and integration we deduce that
\begin{align}
\E \int_{S}\abs{\Pa(b(s))}\abs{\xi_{s}(\omega)}^{2}\;d\mu(s)
  & = \int_{S}\abs{\Pa(b(s))}^{2}\E\abs{\xi_{s}(\omega)}^{2}\;d\mu(s) \label{eq:bias-variance-deco-1}\\
  &= \int_{S}\abs{\Pa(b(s))}^{2}\;d\mu(s) \nonumber
\end{align}
%
For the above identity it is important to have the right hand side
finite; that is, $\Pa\circ b\in \Lmu$.

In the subsequent analysis we shall distinguish the cases of finite
measure~$\mu$, i.e.,\ when~$\mu(S)<\infty$ and the infinite case~$\mu(S)=\infty$.

Plainly,
if the measure~$\mu$ is finite then we have
from Definition~\ref{de:regularization} the uniform bound
$$
\int_{S}\abs{\Pa(b(s))}^{2}\;d\mu(s) \leq
\frac{C_{-1}^{2}}{\alpha^{2}} \mu(S),\quad \alpha>0.
$$
Otherwise, this needs not be the
case as highlights the following
\begin{xmpl}
  \label{xmpl:counter}
Consider the multiplication operator
$$
g := b \cdot f,\quad f\in L_{2}(\real,\lambda),
$$
where
$$b(s) = \begin{cases}
  0, & s<0,\\
  s, & 0 \leq s\leq 1,\\
  1,& s>1,
\end{cases}
$$
with~$\lambda$ denoting the Lebesgue measure on~$\real$.

Let $(\Phi_\alpha)$ be an arbitrary regularization. From
Definition~\ref{de:regularization} we know that $\Pa(1) \to 1$
as~$\alpha\to 0$, and hence there is~$\alpha_{0}>0$ such that~$\Pa(1)
\geq 1/2$ for~$0 < \alpha \leq \alpha_{0}$. Therefore, for each~$s\geq
1$ and $0<\alpha \leq \alpha_0$ we have that~$\Pa(b(s))=\Pa(1) \geq 1/2$, and
$$ \int_\R \abs{\Pa(b(s))}^{2}\;d\l(s) \geq  \int_1^\b \abs{\Pa(b(s))}^{2}\;d\l(s) \geq  \frac{1}{4}(\b-1)
$$
for every $\b>1$ so that the integral $\int_\R \abs{\Pa(b(s))}^{2}\;d\l(s)$ is not finite. Consequently, the multiplication equation with the above $b(\cdot)$ cannot be solved with arbitrary
accuracy (as~$\delta\to 0$) under white noise by using any
regularization.
\end{xmpl}

\begin{xmpl}
  [Lavrent'ev regularization, continued]
  Suppose that~$\mu(S)=\infty$, and that~$b$ vanishes at infinity. For Lavrent'ev regularization we then
  see that
  \begin{align*}
  \int_{S}\abs{\Pa(b(s))}^{2}\;d\mu(s) & \geq
  \int_{\set{s,\ b(s) \leq  \alpha}}\frac 1 {(\alpha + b(s))^{2}}\;d\mu(s)\\
    &\geq \frac 1
  {4\alpha^{2}} \mu\lr{\set{s,\ b(s) \leq \alpha}} = \infty.
  \end{align*}
\end{xmpl}
However, under Assumption~\ref{ass:finite-condition} we have that
$$
\int_{S}\abs{\Pa(b(s))}^{2}\;d\mu(s) = \int_{\set{b>\alpha}}\abs{\Pa(b(s))}^{2}\;d\mu(s),
$$
and this will be finite for functions~$b$ vanishing at infinity.

We recall the decreasing rearrangement~$\bast$ of the multiplier
function~$b$. Since both~$b$ and~$\bast$ share the same distribution
function we can use the transformation of measure formula to  see for any
(measurable) function~$H\colon [0,\norm{b}{\infty})\to \real$ that
$$
\int_{[0,\mu(S))} \abs{H(\bast(t))}^{2} \;d\lambda(t)
= \int_{S}\abs{H(b(s))}^{2} \;d\mu(s)
$$
In particular this holds for spectral cut-off as in
Example~\ref{xmpl:cut-off-def}, used as~$H(s):=
\Pa^{\text{c-o}}(b(s)\chi_{\set{b(s) >\alpha}}$, yielding
\begin{equation}
  \label{eq:measure-trafo-c-o}
\int_{\set{\bast>\alpha}} \frac 1 {\abs{\bast(t)}^{2}} \;d\lambda(t)
= \int_{\set{b > \alpha}}\frac 1 {\abs{(b(s))}^{2}} \;d\mu(s).
\end{equation}
We now observe that from the definition of regularization functions,
see Definition~\ref{de:regularization} we have for arbitrary
regularization~$\Pa$  that~$\Pa(t) \leq
\frac{C_{0}+1}{t}$, and hence that
\begin{align*}
\int_{\set{b > \alpha}}\abs{\Pa(b(s))}^{2} \;d\mu(s) & \leq
                                                       (C_{0}+1)^{2}\int_{\set{b
                                                       > \alpha}}\frac
                                                       1
                                                       {\abs{(b(s))}^{2}}
                                                       \;d\mu(s) \\
& = (C_{0}+1)^{2}\int_{\set{\bast>\alpha}} \frac 1 {\abs{\bast(t)}^{2}} \;d\lambda(t).
\end{align*}
This gives rise to the following
\begin{de}
  [statistical effective ill-posedness]\label{de:effective-ill-posedness}
  Suppose that we are given the function~$b>0$ on the measure
  space~$(S,\mathcal F,\mu)$. For a function~$b$ that vanishes at infinity
  we call the function
  \begin{equation}
    \label{eq:degree-ill-posedness}
    D(\alpha) := \lr{\int_{\set{\bast>\alpha}} \frac 1
    {\abs{\bast(t)}^{2}} \;d\lambda(t)}^{1/2},\quad \alpha>0,
  \end{equation}
  the statistical effective ill-posedness of the operator.
\end{de}
\begin{xmpl}
  [Spectral cut-off, counting measure]\label{xmpl:counting}
  Suppose that~$S=\nat$ and~$\mu$ is the counting measure
  assigning~$\mu(\set{j}) = 1,\ j\in\nat$, and that the function~$j\mapsto
  b(j)$ is non-increasing with~$\lim_{j\to \infty} b(j) = 0$. 
    Then it vanishes at infinity, and
  %
  for each~$\alpha>0$ there will be a maximal finite
  number~$N_{\alpha}$ with~$b(N_{\alpha}) \geq \alpha >
  b(N_{\alpha}+1)$.
  
  In this case the statistical effective ill-posedness evaluates as
  $$
  D(\alpha) =\lr{\sum_{j=1}^{N_{\alpha}}\frac 1 {b_{j}^{2}}}^{1/2},\quad \alpha >0.
  $$
  This corresponds to the 'degree of ill-posedness for statistical
  inverse problems' as given in~\cite{deg}.
  The bias-variance decomposition
  from~(\ref{eq:err-deco-white-noise}) is known to be order optimal,
  cf.~\cite{MR3859257}.
\end{xmpl}
The following bound simplifies the statistical effective ill-posedness,
and can often be used.

\begin{lem}\label{lemma-1}
  Let~$\Pa$ be any regularization. Under
  Assumptions~\ref{ass:distfunction} and~\ref{ass:finite-condition} we
  have that
  $$
  D(\alpha) \leq
  \frac{1}{\alpha}\sqrt{\mu(\set{s:\ b(s) >\alpha})},\ \alpha>0,
  $$
  and
  $$
  \int_{S}\abs{\Pa(b(s))}^{2}\;d\mu(s) \leq
  \frac{C_{-1}^{2}}{\alpha^{2}}\mu\lr{\set{b>\alpha}}
  $$
\end{lem}
\begin{proof}
  The result follows from   Definition~\ref{de:regularization}.
\end{proof}

We summarize the preceding discussion as follows. Suppose that the
measure~$\mu(S)=\infty$, and that assumptions~\ref{ass:distfunction}
and~\ref{ass:finite-condition} hold true. The error
decomposition~(\ref{eq:err-deco-white-noise}) then yields
\begin{equation}
  \label{eq:error-deco-D}
   \E\norm{f - \fad}{\Lmu}^{2} \leq \norm{\Ra(b) f}{\Lmu}^{2} +
   \delta^{2} (C_{0}+1)^{2}D^{2}(\alpha),\quad \alpha>0.
\end{equation}

Using this we obtain the following analog of Proposition
\ref{est-derministic}.

\begin{prop}\label{est-white}
  Suppose that the solution $f$ satisfies the source condition as in
  Definition~\ref{de:source-condition}, and that  a
  regularization~$\Pa$ is chosen with
  qualification~$\varphi$. Suppose, in addition,  that
  Assumptions~\ref{ass:distfunction} and~\ref{ass:finite-condition} hold. Then, for the case of white noise $\xi$,
  $$
  \mathbb{E} \norm{f - \fad}{\Lmu}^2  \leq C_{\varphi}^2\varphi(\alpha) ^2 +\delta^{2} (C_{0}+1)^{2}D^{2}(\alpha),\quad \alpha>0.
  $$
  The \emph{a priori parameter choice}~$\alpha_{\ast}=
  \alpha_{\ast}(\varphi,D,\delta)$ from solving the equation
  \begin{equation}
    \label{eq:a-priori-choice1}
    \varphi(\alpha) = \delta D(\alpha)
  \end{equation}
  yields the error bound
  \begin{equation}
    \label{eq:a-priori-error1}
   \lr{ \mathbb{E} \norm{f - \fad}{\Lmu}^2}^{1/2}  \leq \sqrt
   2\max\set{C_{\varphi},(C_{0}+1)}\varphi(\alpha_{\ast}). 
  \end{equation}
\end{prop}

\begin{proof}
Follows from Definitions~\ref{de:qualification}
and~\ref{de:effective-ill-posedness}, (\ref{eq:bias-variance-deco}),
and~(\ref{eq:bias-variance-deco-1}).
\end{proof}

 \section{Operator equations in Hilbert space}
\label{sec:op-eqns}

As outlined in the introduction the setup of multiplication operators
as analyzed here  is prototypical for
general bounded self-adjoint positive operators~$\Aop\colon H \to H$ mapping in the  (separable) Hilbert
space~$H$ due to the associated Spectral Theorem
(cf.~\cite{MR0150600}), stated as {\bf Fact}. It is an advantage of our focus
on multiplication operators that we can include {\it compact} linear operators and {\it non-compact} ones as well.

\begin{xmpl}
  [Compact operator]\label{xmpl:compact}
It was emphasized in~\cite{MR2211346} that
the case of a compact positive self-adjoint operator~$\Aop$ yields a
multiplier version with~$S=\nat$, $\Sigma=\mathcal{P}(\nat)$, and~$\mu$ being the
counting measure, i.e.,\ $L^2(S, \Sigma,\mu) = \ell^{2}$,  and multiplier~$b:=
\lr{b_{j}}_{j\in\nat}$, where~$b_{j}$ denotes the $j$th eigenvalue
taking into account (finite) multiplicities.  White noise in~$\ell^{2}$
is given by a sequence of i.i.d. random
variables~$\xi_{1},\xi_{2},\dots$ with mean zero and variance one.
\end{xmpl}

In the subsequent discussion we shall highlight the impact of the
previous results, presented for equations with multiplication operator
for specific operator equations with non-compact operator~$\Aop$.

\subsection{Deconvolution}
\label{sec:deconvolution}

Suppose that data~$\yd$ are a real-valued function on $\real$ and given as
\begin{equation}
  \label{eq:convol-base}
  \yd(t) = \lr{r *x}(t) + \delta \eta(t),\quad t\in\real.
\end{equation}
In the above, $\lr{r * x}(t) := \int_{\real} r(u-t) x(u) \; du$ for
$t\in \real $. The noise~$\eta$ is assumed to be symmetric around
  zero and (normalized) weighted white noise~$\eta(t) := w(t) dW_{t},\
  t\in\real$,  with a square integrable weight normalized
  to~$\norm{w}{L^{2}(\real)}=1$.
The goal is to find approximately the function~$x(t),\ t \in \real$, based on noisy
data~$\yd$. This problem is usually called~\emph{deconvolution}.

\subsubsection{Turning to multiplication in frequency space}
\label{sec:frequency}

In order to transfer the deconvolution task into the multiplication
form~(\ref{eq:opeqmult}) with noisy data (\ref{eq:ill-posed-multiplier}) we
 use the Fourier transform to get
\begin{equation}
  \label{eq:F-conv}
  \gd(s) := \hat{{\yd}}(s) = \hat r(s) \hat x(s) + \delta
  \hat\eta(s),\quad s\in\real.
\end{equation}
We make the following assumptions.
First, we  assume that the kernel function~$r \in L^{1}(\real)$ is
non-negative,
symmetric around zero, and that~$u \mapsto r(u),\ u>0$ is
non-increasing. In this case its Fourier transform~$b(s) := \hat
r(s)$ is \emph{non-negative and real valued}. Also, $b\in
C_0(\real)$,
and zero is an accumulation point of
the essential range of $b$. Thereby the corresponding multiplication operator
does not have closed range.  We denote~$f(s) = \hat x(s),\ s\in\real$.
 Then it is easily
checked that the Fourier transform~$\xi(s) := \hat \eta(s)$ is
centered Gaussian, and~$\E
\xi(s) \bar\xi(s^{\prime}) = 0$ whenever~$s\neq s^{\prime}$.
  By the properties of the noise, as described before, the
variance is given as
$$
\E\abs{\xi(s)}^{2} = \int_{\real} \abs{w(u)}^{2} \; du =1.
$$
Thus we arrive at the multiplication equation~(\ref{eq:F-conv})
as in Section~\ref{sec:intro}.

\subsubsection{Relation to reconstruction of stationary time series}
\label{sec:stationary}
Historically, the deconvolution problem was first studied by Wiener
in~\cite{MR0031213}. 
 In that context the solution~$f$ in~(\ref{eq:base-multiplication}) is a
stationary time series~$f_s(\omega)$ with (constant) average signal strength~$S_{f}
:= \E\abs{f(s)}^{2}$. Then we may look for a (real valued)
multiplier~$h(s),\ s\in S$ such that~$f^{\delta}(s):= h(s) \gd(s)$ is a MISE
estimator, i.e.,\ it minimizes (point-wise) the functional
\begin{equation}
  \label{eq:mise}
  \E_{f}\E_{\xi}\abs{f^{\delta}(s) - f(s)}^{2},\quad s\in S.
\end{equation}
Assuming that the noise~$\xi(s)$ is independent from the signal~$f(s)$
the above minimization problem can be rewritten as
$$
 \E_{f}\E_{\xi}\abs{f^{\delta}(s) - f(s)}^{2} = \abs{1 -
   h(s)b(s)}^{2}S_{f} + \abs{h(s)}^{2}\E\abs{\xi(s)}^{2},\ s\in S.
 $$
 The minimizing function~$h(s)$ (in the general complex valued case,
 and with~$\bar b$ denoting the complex conjugate to $b$) has the form
 \begin{equation}
   \label{eq:h-minimal}
   h(s) := \frac{\bar b(s) S_{f}}{\abs{b(s)}^{2} S_{f} + \delta^{2}} =
   \frac{\bar b(s)}{\abs{b(s)}^{2} + \frac{\delta^{2}}{s_{f}}}.
 \end{equation}
 This approach results in the classical \emph{Wiener
  Filter}, see~\cite{MR0031213}. 
 Notice that the quotient~$\sqrt{S_{f}}/\delta$ is the
 signal-to-noise ratio, a constant which is unknown, thus
 replacing~$\delta^{2}/S_{f}$ by $\alpha$ we arrive at the
 reconstruction formula
$$
\fd (s) := \frac{\bar b(s)}{ \alpha + \abs{b(s)}^{2}} \gd(s),\quad s\in\real,
$$
being the analog to Tikhonov regularization.

However, since here we assume~$b$ to be real and positive, one may propose the Lavrent'ev
approach resulting in
\begin{equation}
  \label{eq:conv-lavr}
  \fd(s) := \frac{1}{\alpha + b(s)}\gd(s),\quad s\in\real,
\end{equation}
and hence to the regularization scheme~$\Pa(t) := 1/(\alpha + t),\
\alpha >0,\ t>0$ as introduced in~\S~\ref{sec:assumptions}. Other
regularization schemes also apply.

\subsection{Final value problem}
\label{sec:fvp}
\def\<{\langle}
\def\>{\rangle}
In the final value problem (FVP), also known as backward heat conduction problem associated with the
heat equation
\begin{equation}\label{FVP-1}
\frac{\partial }{\partial t}u(x, t) = c^2 \Delta u(x,t),\quad  x\in \Omega,\, \, 0<t<\tau,
\end{equation}
one would like to determine the initial temperature
$f_0:= u(\cdot, 0)$,  from the knowledge of the final temperature
$f_\tau:= u(\cdot, \tau)$ . Here, the domain~$\Omega$ is in~$\mathbb{R}^d$.
This problem is known to be ill-posed.


It can be considered as an
operator equation with multiplication operator. A similar FVP was
considered in the recent study~\cite{2019arXiv190711076T}.

\subsubsection{$\Omega=\mathbb{R}^d$}
\label{sec:rd}

In this case,
on taking  Fourier transform of the functions on both sides of equation (\ref{FVP-1}) , we obtain
$$
\frac{\partial }{\partial t}\hat u(s, t) = - c^2 |s|^2  \hat u(s, t),\quad  s\in \mathbb{R}^d,\, \, 0<t<\tau.
$$
For each fixed $s\in \Omega$, the above equation is an ordinary differential  equation, and hence the solution  $\hat u(s, t)$ is given by
$$\hat u(s, t) = e^{-c^2t |s|^2}\hat f_0(s),\quad s\in \mathbb{R}^d,$$
where $f_0(x):=u(x, 0),\, x\in \mathbb{R}^d$. In particular, with $t=\tau$,  we have
$$\hat u(s, \tau) = e^{-c^2\tau |s|^2}\hat f_0(s),\quad s\in \mathbb{R}^d.$$
Taking
$$f(s):= \hat f_0(s),\quad g(s):= \hat f_\t(s),\quad b(s):= e^{-c^2t |s|^2},$$
the above equation takes the form
\begin{equation}\label{FVP-2}
  b(s) f(s) = g(s),\quad s\in \mathbb{R}^d.
\end{equation}
Here, one may assume that the actual data $g(\cdot)$
belongs to $\in L^2(\mathbb{R}^d)$. The problem is to determine the function $f(\cdot)\in  L^2(\mathbb{R}^d)$ satisfying  {\it multiplication} operator equation (\ref{FVP-2}).

We may recall that the map $h\mapsto \hat h$ is a bijective linear isometry from $L^2(\R^d)$ into itself. Therefore, if~$f_\t^\d$ is a noisy data, then
$$\|f_\t-f_\t^\d\|_{L^2(\R^d)} = \|g-g^\d\|_{L^2(\R^d)},$$
where $g^\d:= \hat{f_\t^\d}$. Hence, if $f^\d$ is an approximate solution corresponding to the noisy data $g^\d$, and if $f_0^\d$ is the inverse Fourier transform of $f^\d$, then we have
$$\|f_0-f_0^\d\|_{L^2(\R^d)} = \|f-f^\d\|_{L^2(\R^d)}.$$
Thus, in order to obtain the error estimates for the regularized solutions corresponding to noisy measurements $f_\t^\d$, it is enough to consider the noisy equation as in (\ref{eq:base-multiplication}), that is,
$$
g^\d(s) = b(s) f(s) + \d \xi(s),\quad s\in \mathbb{R}^d.
$$
%

\subsubsection{$\Omega$ is a bounded domain in $\mathbb{R}^d$}
\label{sec:bounded}
For the purpose of illustration, let us assume that the the temperature is kept  at $0$ at the boundary of $\Omega$, that is,
$$u(x, t) = 0\quad\hbox{for}\quad x\in \partial \Omega.$$
Then the solution of the equation (\ref{FVP-1})  along with the the initial condition
$$u(x, 0) = f_0 (x),\quad x\in \Omega,$$
is given by (see~\cite[\S~4.1.2]{MR2521497})
$$u(x, t) = \sum_{n=1}^\infty e^{-c^2\lambda_n^2t}\< f_0, v_n\>  v_n(x).$$
Here $(\lambda_n)$ is a non-decreasing sequence of non-negative real numbers such that $\lambda_n\to \infty$ as $n\to\infty$ and
$(v_n)$ is an orthonormal sequence of functions in $L^2(\Omega)$.  In fact, each $\lambda_n$ is an eigenvalue of the operator~$(-\Delta)$ with corresponding  eigenvector $v_n$.
For $t=\tau$, taking  $f_\tau := u(\cdot, \tau)$, we have
$$f_\tau(x)  = \sum_{n=1}^\infty e^{-c^2\lambda_n^2\tau}\< f_0, v_n\>  v_n(x).$$
Equivalently,
\begin{equation}\label{FVP-3}
\<f_\tau, v_n\>=  e^{-c^2\lambda_n^2\tau} \<f_0, v_n\>,\quad n\in \mathbb{N}.
\end{equation}
Writing
$$ g:=(\<f_\tau, v_n\>),\quad f:=(\<f_0, v_n\>),\quad b:= (e^{-c^2\lambda_n^2\tau}),$$
the system of equations in (\ref{FVP-3}) takes the form a multiplication operator equation
\begin{equation}\label{FVP-3new}
b(n) f(n) = g(n),\quad n\in \mathbb{N},
\end{equation}
where $g$ and $f$ are in $\ell^2(\N)$ and $b$ is in $c_0(\N)$, the space all null sequences.

As in~\S~\ref{sec:rd}, we have
$$\|f_\t-f_\t^\d\|_{L^2(\O)} = \|g-g^\d\|_{\ell^2(\N)}\quad\hbox{and}\quad
\|f_0-f_0^\d\|_{L^2(\O)} = \|f-f^\d\|_{\ell^2(\N)},$$
where $g^\d\in \ell^2(\N)$ and $f_0^\d\in L^2(\O)$ are constructed from the bijective linear isometry
$h\mapsto (\<h, v_n\>)$ from $L^2(\O)$ onto $\ell^2(\N)$, that is,
$$g^\d(n):= \<f_\t^\d, v_n\>\quad\hbox{and}\quad f_0^\d:=\sum_{n=1}^\infty \<f^\d, v_n\> v_n.$$


\end{document}